\pdfoutput=1

\documentclass[11pt]{article}
\usepackage[utf8]{inputenc}
\usepackage{amsmath}
\usepackage{amsfonts}
\usepackage{amsthm}
\usepackage{amssymb}
\usepackage{mathtools}
\usepackage[T1]{fontenc}
\usepackage{pst-node}
\usepackage{tikz-cd}
\usepackage{caption}
\usepackage{bbm}
\usepackage{comment}
\usepackage{dynkin-diagrams}
\usepackage{comment}

\usepackage{geometry}
\theoremstyle{plain}
\newtheorem{thm}{Theorem}[section]
\newtheorem{lem}[thm]{Lemma}
\newtheorem{prop}[thm]{Proposition}
\newtheorem{cor}[thm]{Corollary}

\theoremstyle{definition}
\newtheorem{defn}[thm]{Definition}

\theoremstyle{remark}
\newtheorem*{rem}{Remark}

\newcommand{\co}{\mathbb{C}}

\newcommand{\ze}{\mathbb{Z}}

\newcommand{\na}{\mathbb{N}}

\newcommand{\Hom}{\mathrm{Hom}}

\newcommand{\Aut}{\mathrm{Aut}}

\renewcommand{\a}{\alpha}
\renewcommand{\b}{\beta}

\renewcommand{\d}{\delta}
\newcommand{\D}{\Delta}

\renewcommand{\L}{\Lambda}
\newcommand{\s}{\sigma}

\renewcommand{\t}{\tau}
\renewcommand{\Tilde}{\widetilde}

\title{Finitely presented simple groups and measure equivalence}
\author{Antonio López Neumann}

\date{}
  
\begin{document}

\maketitle

\begin{abstract}
    We exhibit explicit infinite families of finitely presented, Kazhdan, simple groups that are pairwise not measure equivalent. These groups are lattices acting on products of buildings. We obtain the result by studying vanishing and non-vanishing of their $\ell^2$-Betti numbers.
    
    \vspace*{2mm} \noindent{2020 Mathematics Subject Classification: } 20E32, 20E42, 20F05, 20F55, 20J05, 20J06, 22E41, 28D15.

%

%

\vspace*{2mm} \noindent{Keywords and Phrases: } Infinite simple groups, finite presentation, $L^2$-Betti numbers, measure equivalence, Coxeter groups, buildings, cohomological finiteness.

\end{abstract}

\section{Introduction}

Infinite finitely presented simple groups are rare in geometric group theory. To this date, few examples are known: Burger-Mozes groups acting on products of trees \cite{burger-mozes}, non-affine irreducible Kac-Moody lattices over a finite field acting on twin buildings \cite{caprace-remy-simple} and variants of Thompson groups \cite{rover}.

Naturally, we study these groups not up to isomorphism but up to equivalence relations that are relevant to the theory, like quasi-isometry or measure equivalence. Some results have been obtained on the quasi-isometry side. All Burger-Mozes groups are quasi-isometric since they act properly and cocompactly on products of trees, which are bi-Lipschitz equivalent. It is shown in \cite{caprace-remy-non-distortion} that there is an infinite family of non-affine irreducible Kac-Moody lattices that are pairwise not quasi-isometric. Using cohomological finiteness properties, the same result is shown in \cite{SWZ} for Röver-Nekrashevych variants of Thompson groups.

We obtain the same result as in \cite{caprace-remy-non-distortion} and in \cite{SWZ} but for measure equivalence. Namely:

\begin{thm}\label{infinite_family}
There are infinitely many measure equivalence classes containing finitely presented, Kazhdan, simple groups. These groups are Kac-Moody lattices over finite fields with well-chosen non-affine Weyl groups.
\end{thm}

Theorem \ref{infinite_family} is proven by studying the sequence of $\ell^2$-Betti numbers $(\b^k)_{k \in \na}$ of some of these lattices. Indeed, it was shown by Gaboriau that the sequence of $\ell^2$-Betti numbers of discrete countable groups is invariant by measure equivalence up to proportionality \cite[6.3]{gaboriau}. More precisely, we find for every $n \in \na$, a finitely presented, Kazhdan, simple (Kac-Moody) group $\L_n$ such that $\b^{d(n)}(\L_n) > 0$ and $\b^{k}(\L_n) = 0$ for $k > d(n)$, where $d(n) \to \infty$. Vanishing of $\b^k$ is a coarse equivalence invariant for any $k \in \na$ \cite{sauer-schrodl}, so this more precise statement proves the analogue of Theorem \ref{infinite_family} for coarse equivalence, in particular we recover the corresponding result of \cite{caprace-remy-non-distortion} concerning quasi-isometry classes. The result concerning coarse equivalence is not new, as it is implicit in \cite{caprace-remy-non-distortion}: they state that their result may alternatively be obtained using asymptotic dimension, which is a coarse equivalence invariant.

Petersen provides a first family of finitely generated simple (Kac-Moody) groups having non proportional sequences of $\ell^2$-Betti numbers \cite[6.8]{petersen} using an example by Dymara and Januszkiewicz \cite[8.9]{dymarajanuszkiewicz}. Unfortunately, the groups in question are not Kazhdan and it is not known if they are finitely presented. A conjecture \cite{abramenko-gates} says that such groups should never be finitely presented, due to the presence of $\infty$'s in the Coxeter diagram of the Weyl group.

Our proof follows the strategy outlined by Petersen but for families of Kac-Moody groups enjoying better properties, we describe the strategy in two steps. First, we use Dymara and Januszkiewicz's formula for $L^2$-Betti numbers of locally compact groups acting on some buildings (\cite[Section 8]{dymarajanuszkiewicz}, here Theorem \ref{formuleBetti}, see also \cite{grinbaum-reizis-oppenheim} for weaker bounds on the thickness). These groups include (products of) complete Kac-Moody groups, which can be seen as the ambient spaces of Kac-Moody lattices. Second, we use a result by Petersen \cite[5.9]{petersen} relating the sequence of $\ell^2$-Betti numbers of a lattice to the sequence of $L^2$-Betti numbers of the ambient topological group. This allows us to prescribe vanishing and non-vanishing of some $\ell^2$-Betti numbers of Kac-Moody lattices.

The formula for $L^2$-Betti numbers of complete Kac-Moody groups is explicit but still hard to manipulate. Roughly, the formula splits into two parts: a topological one and a representation theoretic one. As explained in \cite{dymarajanuszkiewicz}, the behaviour of the representation theoretic part is well understood. Thus, the difficulty comes from understanding the topological part. More precisely, one has to compute the cohomology of some subcomplexes of the Davis chamber which encode the combinatorial complexity of the groups we consider. The Davis chamber, as explained later, can be constructed entirely from the Weyl group of the building. This reduces our study of $L^2$-Betti numbers to a purely combinatorial study of Coxeter groups. Indeed, this reduces the proof of Theorem \ref{infinite_family} to finding a sequence of 2-spherical, non-affine Coxeter systems with unbounded virtual cohomological dimension. We compute the cohomology of some of these subcomplexes, precisely enough to give a non-vanishing criterion of an $L^2$-Betti number (in high degree) for groups acting on buildings having a certain Weyl group. The condition we require for the Weyl group is that the canonical generating set may be partitioned into two parts, such that one part generates a finite Coxeter group and the other generates an affine Coxeter group. This is enough to obtain Theorem \ref{infinite_family}.

We give explicit families of groups as in Theorem \ref{infinite_family} where we compute all of these cohomology spaces, thus simplifying the formula for $L^2$-Betti numbers of the corresponding complete Kac-Moody groups and Kac-Moody lattices. This formula can be stated explicitly by following the proof, we do not do it because what matters to us is to determine for which degrees the $\ell^2$-Betti numbers vanish.

Now we present the contents of the sections.

Section 2 introduces the necessary background for the rest of the paper. We first define the classes of simplicial complexes and groups discussed in \cite{dymarajanuszkiewicz}, as well as some combinatorial properties of the Davis chamber. We introduce Kac-Moody groups and list the results making them interesting for us. We then state the main theorems from \cite{dymarajanuszkiewicz}, especially the formula for $L^2$-Betti numbers of groups acting on buildings.

Section 3 deals with the cohomology of some of the subcomplexes $D_\s$ appearing in the formula of $L^2$-Betti numbers given in Section 2. Using nerves, we first give a vanishing criterion for this cohomology, slightly simplifying this formula. We then give a non-vanishing criterion under a condition on the Weyl group, and apply it to Kac-Moody lattices. This proves Theorem \ref{infinite_family}.

Section 4 gives concrete examples of Coxeter groups where results from Section 3 compute all $L^2$-Betti numbers of the corresponding complete and discrete Kac-Moody groups.

Section 5 addresses cohomological finiteness properties of the simple Kac-Moody lattices we study. Our arguments point out that it should not be possible to show the results of \cite{caprace-remy-non-distortion} using cohomological finiteness criteria as in \cite{SWZ}.

\tableofcontents

\paragraph{Acknowledgements} I would like to thank Bertrand Rémy and Marc Bourdon for many useful discussions, as well as Damien Gaboriau for pointing out to me that vanishing of an $\ell^2$-Betti number is a quasi-isometry invariant. I also thank the reviewer for multiple comments that improved the exposition of the proofs.

\section{$L^2$-Betti numbers of groups acting on buildings}

This section is mainly a review of some parts of \cite{dymarajanuszkiewicz}. First, we introduce the classes of simplicial complexes of \cite[Section 1]{dymarajanuszkiewicz}. Let $X$ be a \textit{purely $n$-dimensional} countable simplicial complex (every simplex is a face of an $n$-dimensional simplex). Top dimensional simplices in $X$ will be called \textit{alcoves}. Let $\Aut(X)$ be the group of simplicial automorphisms equipped with the compact-open topology and $G$ be a closed subgroup of $\Aut(X)$. We consider the following properties on the pair $(X,G)$.

$\mathcal{B}1$ $0$-dimensional links in $X$ are finite.

$\mathcal{B}2$ Links of dimension $\geq 1$ in $X$ are gallery-connected: for any two alcoves in such a link, there exists a path of alcoves connecting them, such that two consecutive elements meet on a face of codimension 1. 

$\mathcal{B}3$ All the links in $X$ are either finite or contractible (including $X$ itself).

The three properties listed above deal with the space $X$ only. The following condition is the only one that considers both the group $G$ and the space $X$. Roughly it requires the group to have the good size: it is big enough to act transitively on alcoves, and small enough to have a fundamental domain of maximal dimension.

$\mathcal{B}4$ The group $G$ acts transitively on the set of alcoves of $X$ and the quotient map $X \to X/G$ restricts on an isomorphism on each alcove.

The next two properties are spectral conditions on the Laplacian that are fundamental in \cite{dymarajanuszkiewicz} to prove the results presented in Section 2.3.

$\mathcal{B}_\d$ Links of dimension 1 are compact and the nonzero eigenvalues of the Laplacian on 1-dimensional links are $\geq 1 - \d$.

$\mathcal{B}_\d^*$ The spectrum of the Laplacian on 1-dimensional links in $X$ is a subset of $\{ 0 \} \cup [1-\d, +\infty[$.

We now begin a brief discussion on buildings, since they are the main source of examples of spaces satisfying (most of) these conditions. A standard reference for buildings is \cite{brown}. They are simplicial complexes obtained by gluing subcomplexes, called \textit{apartments}, under two incidence conditions: any two simplices lie in an apartment and their position is independent of the apartment. Each apartment is a copy of the same Coxeter complex, a purely dimensional simplicial complex with a simply transitive action (on its set of alcoves) of a Coxeter group called the \textit{Weyl group} of the building. The number of alcoves containing a given face of codimension 1 is called the \textit{thickness} (of that face) and for buildings we want it to be $\geq 3$ for all such faces. The buildings we are interested in have finite thickness, they satisfy $\mathcal{B}1$.

\subsection{Davis complex and Davis chamber}

Let $X$ be a simplicial complex satisfying $\mathcal{B}3$.

\begin{defn}
Let $X'$ be the first barycentric subdivision of $X$. The \textit{Davis complex} $X_D$ of $X$ is the subcomplex of $X'$ generated by the barycenters of simplices of $X$ with compact links. 
\end{defn}

This definition is interesting for two reasons. The first, is that $X_D$ is a deformation retract of $X$ \cite[1.4]{dymarajanuszkiewicz} which has the same automorphism group $\Aut(X)$, but the action of the latter becomes proper over $X_D$. The second reason stems from buildings, and is summarized in the following proposition.

\begin{prop}
Let $X$ be a building. Then the link of every simplex is a building. \\
Suppose $X$ is a non-compact building. Then $X_D$ can be endowed with a $CAT(0)$ metric. In particular, $X$ and $X_D$ are contractible.
\hfill $\square$
\end{prop}

The first assertion is \cite[IV.1 Prop 3]{brown} and the second one can be found in \cite{daviscat0}.
In particular, this shows that buildings satisfy $\mathcal{B}2$ and $\mathcal{B}3$. 

Property $\mathcal{B}4$ is not always satisfied in this setting since one can consider buildings without automorphisms. However, if $G$ is a group with $BN$-pair and $X$ is the building constructed from it, then the pair $(X,G)$ satisfies $\mathcal{B}4$.


If a pair $(X, G)$ satisfies $\mathcal{B}3$ and $\mathcal{B}4$, the intersections $D = X_D \cap \D$ are simplicially isomorphic for any top dimensional simplex $\D$. We call such an intersection a \textit{Davis chamber} of $X$ and we denote it $D$. We can see $D$ as a cone over $D \cap \partial \D$ with apex the barycenter of $\D$, thus $D$ is contractible.

In the case of buildings, the Davis chamber may also be constructed from its Weyl group. The construction is equivalent to the previous one if the building comes from a BN-pair.

Let $(W,S)$ be a Coxeter system with $|S| = n+1$ and let $\D$ be the standard simplex of dimension $n$. We associate to each codimension 1 face of $\D$ a generator $s \in S$. This choice determines for each face $\s$ in $\D$ a parabolic subgroup $W_\s$ of $W$, where $W_\s = W_J = \langle J \rangle$ and $J \subseteq S$ is the set of generators corresponding to codimension 1 faces containing $\s$.

\begin{defn}
Let $\D'$ be the first barycentric subdivision of $\D$. We define the \textit{Davis chamber} of $W$ as the subcomplex $D = D_W$ of $\D'$ generated by the barycenters of faces $\s$ in $\D$ whose corresponding parabolic subgroup $W_\s$ is finite.
\end{defn}

The next definition measures in a certain sense how many simplices we must remove from $\D$ to obtain $D$. 

\begin{defn}
Let $(W, S)$ be a Coxeter system. We say $(W, S)$ is \textit{$k$-spherical} if for all $J \subseteq S$ with $|J| \leq k$, the parabolic subgroup $W_J = \langle J \rangle$ of $W$ is finite.
\end{defn}

This definition will appear often as a hypothesis on the Weyl group for the results we will obtain. We will discuss it in more detail in Section 5.

If $\s$ and $\t$ are faces of $\D$, then $\s \subseteq \t$ is equivalent to $W_\s \supseteq W_\t$. To describe all faces $\s$ of $\D$ whose corresponding subgroup $W_\s$ is finite, we have to identify those corresponding to maximal finite parabolic subgroups of $W$.

If $\s$ is a face of codimension $k$ of $\D$, then the subcomplex of $\D'$ generated by the barycenters of faces containing $\s$ is a (simplicial subdivision of a) cube of dimension $k$. Each maximal finite parabolic subgroup $W_J$ ($J \subseteq S$) corresponds to a cube of dimension $|J|$ in the Davis chamber.

The Davis chamber $D$ is then obtained as follows. We start with the disjoint union of these cubes, and then we glue some of their faces: if $W_I$ and $W_J$ are two maximal finite parabolic subgroups of $W$, the intersection of their corresponding cubes in $\D'$ is the cube corresponding to $W_{I\cap J}$.

This gives an equivalent construction of the Davis chamber that is independent of $\D$. From each parabolic subgroup $W_J$ ($J\subseteq S$) of $W$ we define the flag complex of parabolic subgroups $\{ W_{J'}, J' \subseteq J \}$ contained in $W_J$ ordered by inclusion. We construct $D$ as the union of all flag complexes of maximal finite parabolic subgroups of $W$, where we glue the complexes corresponding to $W_I$ and $W_J$ over the flag complex of $W_{I \cap J}$.

For $\s \subseteq \D$, let $\D_\s$ be the union of faces not containing $\s$ and $D_\s = D \cap \D_\s$. Notice $\D_\s$ is always a union of $(n-1)$-dimensional simplices of $\D$ and that $\D_\s = \partial \D$ if and only if $\s = \D$. More precisely, if $\s \subseteq \D$ corresponds to the parabolic subgroup $W_J$, then $\D_\s$ is exactly the union of codimension 1 faces of $\D$ corresponding to the generators $s_j$ for $j \in  S \setminus J$.

\subsection{Kac-Moody groups}

The family of simple groups we want to exhibit in Theorem \ref{infinite_family} are Kac-Moody groups. 
We introduce them following the presentation of \cite[Appendix TKM]{dymarajanuszkiewicz} and then list the properties that make them interesting to us.

\begin{defn}
A \textit{Kac-Moody datum} is the data $(I, \L, (\a_i)_{i \in I}, (h_i)_{i \in I}, A)$ of: \\
1. A finite set $I$. \\
2. A finitely generated abelian free group $\L$. \\
3. Elements $\a_i \in \L, i \in I$. \\
4. Elements $h_i \in \L^\vee = \Hom (\L, \ze), i \in I$. \\
5. A \textit{generalized Cartan matrix} $(A_{ij})_{i,j \in I}$ given by $A_{ij} = \langle \a_i, h_j \rangle$, satisfying
\begin{equation*}
    A_{ii} = 2, \textit{ if } i \neq j \textit{ then } A_{ij} \leq 0 \textit{ and } A_{ij} = 0 \textit{ if and only if } A_{ji} = 0.
\end{equation*}
\end{defn}

From a Kac-Moody datum (or merely from a generalized Cartan matrix) one can define a Coxeter matrix $M = (m_{ij})_{i,j \in I}$ as follows:
\begin{equation*}
    m_{ii} = 1 \textrm{ and for } i \neq j, m_{ij} =2,3,4,6 \textrm{ or } \infty \textrm{ as } A_{ij}A_{ji} = 0,1,2 ,3 \textrm{ or is } \geq 4, \textrm{ respectively.}
\end{equation*}
We consider the Coxeter group $W$ associated to this matrix:
 \begin{equation*}
     W = \langle r_i \; | \; (r_i r_j)^{m_{ij}}= 1 , \textrm{ for } m_{ij} \neq \infty \rangle.
 \end{equation*}

If a Kac-Moody datum is fixed, Tits defines a group functor associating to each field (or commutative ring in general) $k$ a group $\L(k)$ \cite{tits}. The group $\L(k)$ has two BN-pairs $(B_+, N)$ and $(B_-, N)$ such that their Weyl groups $B_{\pm} / (B_\pm \cap N)$ are isomorphic to the group $W$ coming from the generalized Cartan matrix. 

These BN-pairs define two buildings $X_+$ and $X_-$ of thickness $|k| + 1$ and Weyl group $W$ (therefore the dimensions of these buildings is $|I| -1$), such that $\L(k)$ acts transitively on their sets of chambers \cite[Appendix TKM]{dymarajanuszkiewicz}. These buildings are simplicially isomorphic, we denote them $X$ when it is not necessary to distinguish them. Denote by $G_\pm$ the completion of $\L(k)$ in $\Aut(X_\pm)$ with respect to the compact-open topology.

The following theorem summarizes the properties that justify the study of Kac-Moody groups for our purposes.

\begin{thm}\label{Kac-Moody}
Let $\L$ be a Kac-Moody group over $\mathbb{F}_q$ with Weyl group $W$. Then $\L$ is finitely generated. Moreover: \\
1. The covolume of $\L$ in $G_+ \times G_-$ (diagonally injected) is $W(\frac{1}{q})$, where
\begin{equation*}
    W(t) = \sum_{w \in W} t^{l(w)}.
\end{equation*}
In particular for $q > |I|$, the group $\L$ is a lattice in $G_+ \times G_-$. \\
2. If $W$ is non-affine, irreducible and $\L$ is a lattice in $G_+ \times G_-$, then $\L / Z(\L)$ is simple, where $Z(\L)$ is the center of $\L$. \\
3. If $q \geq 4$ and all the entries of the Coxeter matrix are finite (i.e. the Weyl group is $2$-spherical), then $\L$ is finitely presented.
\end{thm}

\begin{proof}
Assertions 1 and 2 are respectively Proposition 2 and Theorem 20 in \cite{caprace-remy}. Assertion 3 is a simplified version of the main corollary in \cite{abramenko-muhlherr}.
\end{proof}

\begin{rem}
In what follows we will systematically consider center-free Kac-Moody groups. This can always be done without further assumptions on the generalized Cartan matrix by choosing the adjoint root datum, where $\L$ is generated by the $\a_i$'s.
\end{rem}

\subsection{Cohomology and $L^2$-Betti numbers of groups acting on buildings}

Dymara and Januszkiewicz state their results for classes $\mathcal{B}t$ and $\mathcal{B}+$ of pairs $(X,G)$ groups acting on simplicial complexes. The class $\mathcal{B}t$ are pairs $(X,G)$ satisfying $\mathcal{B}1 - 4$ and $\mathcal{B}_{\frac{13}{28^n}}^*$, the class $\mathcal{B}+$ are pairs $(X,G)$ satisfying $\mathcal{B}1 - 4$ and $\mathcal{B}_{\frac{13}{28^n}}$. For groups with a BN-pair and their associated buildings, the class $\mathcal{B}t$ corresponds to large minimal thickness and $\mathcal{B}+$ corresponds to large minimal thickness and finiteness of all entries of the Coxeter matrix \cite[1.7]{dymarajanuszkiewicz}, that is, 2-sphericity of its Weyl group. 

In particular, a complete Kac-Moody group over a finite field $\mathbb{F}_q$ is in $\mathcal{B}t$ for large $q$, and is in $\mathcal{B}+$ if all the entries of its Coxeter matrix are finite.

We now mention three important results of \cite{dymarajanuszkiewicz}. The initial motivation of \cite{dymarajanuszkiewicz} is to find examples of Kazhdan groups. The first theorem we state addresses this question and gives a criterion for the vanishing of the first and also higher cohomology groups for pairs in $\mathcal{B}+$ with a finiteness condition.

\begin{thm}\label{cohomG=0} \cite[5.2]{dymarajanuszkiewicz} Let $(X, G)$ be in $\mathcal{B}+$. Suppose the links of $X$ of dimensions $1, \ldots, k$ are compact. Then for any unitary representation $(\rho, V)$ of $G$, we have:
\begin{equation*}
    H^{i}_{ct}(G, \rho) = 0 \qquad \textrm{ for } i = 1, \ldots, k.
\end{equation*}
\end{thm}

For buildings of finite thickness, compactness of all links of dimension $\leq k$ is equivalent to having a $(k+1)$-spherical Weyl group. This combinatorial condition is necessary for cohomological vanishing: one can consider the group $D_\infty$ acting simplicially on a tree $X$ with edge set $E$, the induced action on $L^2(E)$ does not have a fixed point. Thus we have a space $X$ not verifying the condition of the theorem (the only link of dimension 1 is $X=Lk(\emptyset)$) and $H^{1}_{ct}(G, L^2(E))  \neq 0$.

The second theorem we mention is a formula for cohomology spaces of groups in $\mathcal{B}+$ with values in unitary representations.

\begin{thm}\cite[7.1]{dymarajanuszkiewicz}
Suppose either that \\
the pair $(X, G)$ is in $\mathcal{B}+$ and $(\rho, V)$ is a unitary representation of $G$, or that, \\ 
the pair $(X, G)$ is in $\mathcal{B}t$ and $(\rho, V)$ is a subrepresentation of $\bigoplus^\infty L^2(G)$. \\ 
Then
\begin{equation*}
     H^{*}_{ct}(G, \rho) = \bigoplus_{\s \subseteq \D} \Tilde{H}^{*-1}(D_\s; V^{\s}).
\end{equation*}
\end{thm}

We draw the attention on the right hand side depending on the classical cohomology groups of the spaces $D_\s$ defined at the end of Section 2.1. Theorem \ref{cohomG=0} is proven in \cite[Section 7]{dymarajanuszkiewicz} from this formula by noticing that the cohomology $H^*(D_\s)$ is the cohomology $H^*(\mathcal{U})$ of a simple covering $\mathcal{U}$ of $D_\s$. The combinatorial condition of Theorem \ref{cohomG=0} on the Weyl group implies vanishing of $H^*(\mathcal{U})$ in low degrees for all $\s \subseteq \D$. The proofs given in Section 3 are in the same spirit: under combinatorial conditions on the Weyl group, we compute $H^*(D_\s)$ for some $\s \subseteq \D$. The difference is that instead of considering all simplices $\s \subseteq \D$ and obtaining a partial description of the cohomology of $D_\s$, we pick particular $\s \subseteq \D$ for which we can fully describe the spaces $H^*(D_\s)$.

The last theorem we mention is the starting point for this paper, though the expression we present is obtained directly from the previous theorem. It is a formula for $L^2$-Betti numbers, as defined in \cite[4.1]{sauer}, of groups acting on buildings of finite thickness in $\mathcal{B}t$. As said before, groups in this class include complete Kac-Moody groups over finite fields of large cardinality.

\begin{thm}\label{formuleBetti} \cite[8.5]{dymarajanuszkiewicz}
Let $(X, G)$ be a building in $\mathcal{B}t$ of thickness $q +1$. Then the $L^2$-Betti numbers of $G$ are given by
\begin{equation*}
    \b^k(G) = \sum_{\s \subseteq \D} \dim_\co \Tilde{H}^{k-1}(D_\s; \co) \cdot \dim_G L^2(G)^{\s}.
\end{equation*}
Moreover, the sum can be taken for $\s$ with compact links.
\end{thm}

The fundamental observation by Dymara and Januszkiewicz is that for $q > 2^n$, $\dim_G L^2(G)^{\s} > 0$ for all $\s \subseteq \D$ with compact links \cite[p. 612]{dymarajanuszkiewicz}. Thus, for $q$ large enough, the problem of determining whether $\b^k(G)$ is zero or not reduces to determining if there exists $\s \subseteq \D$ such that $ \Tilde{H}^{k-1}(D_\s; \co) \neq 0$ or not. In the end, this reduces to the study of the combinatorics of the Weyl group. In the next two sections we study the topology of $D_\s$ via the combinatorics of the Weyl group of the building. 

\begin{rem}
1. The techniques of \cite{dymarajanuszkiewicz} are reformulated in \cite{kassabov-subspace} in terms of angles between subspaces. This interpretation allowed Grinbaum-Reizis and Oppenheim \cite{grinbaum-reizis-oppenheim} to recover the previous results from \cite{dymarajanuszkiewicz} under weaker thickness bounds. For (thick) affine buildings, the thickness bound disappears and thus they recover a classical theorem by Casselman \cite{casselman} using geometric methods. This does not improve the main statement of this article, but allows to slightly extend the range of our examples.

2. Dymara and Januszkiewicz apply their results to Kac-Moody groups whose Weyl group $W_P$ is the right-angled Coxeter group defined by the intersection relations of the faces of codimension 1 of a polytope $P$ of dimension $n$ \cite[8.9]{dymarajanuszkiewicz}. Most of the associated Kac-Moody lattices are non-affine and irreducible, hence simple. When $P$ is dual to a triangulation of a sphere (up to considering a barycentric subdivision), it is shown that all $L^2$-Betti numbers of the completions of such groups vanish except in degree $n$. As stated by \cite{petersen}, this gives a first example of an infinite family of finitely generated simple groups with non-proportional sequences of $\ell^2$-Betti numbers. Petersen says such groups should often be finitely presented. Unfortunately, the Coxeter matrices of their Weyl groups have $\infty$ entries, so it is not known whether these groups are finitely presented or not, and a conjecture \cite{abramenko-gates} says that such groups should never be finitely presented. The examples we will give here are finitely presented in view of Theorem \ref{Kac-Moody}, Assertion 3.
\end{rem}

\section{Cohomology of subcomplexes of the Davis chamber}

In this section we elaborate on the formula of Theorem \ref{formuleBetti} for $L^2$-Betti numbers of groups acting on buildings. We study the contribution of the topological part of this formula to obtain first a vanishing criterion, slightly simplifying the formula, and then a non-vanishing criterion.
We then apply this non-vanishing criterion to Kac-Moody lattices and prove Theorem \ref{infinite_family}.

\subsection{Cohomology of nerves: vanishing and non-vanishing}

In what follows a contractible space is non-empty. Let $X$ be a finite simplicial complex covered by a finite family of subcomplexes $\mathcal{U} = \{ X_i , i \in I\}$. The \textit{nerve} $N(\mathcal{U})$ of the cover $\mathcal{U}$ is the simplicial complex whose $k$-simplices are the subsets $J \subseteq I$ with $|J| = k+1$ such that the intersection $\bigcap_{j \in J} X_j$ is non-empty. When $J \subseteq J'$ we have an inclusion of the corresponding simplices. The following result, kwown as the Nerve Lemma, allows us to compute cohomology using nerves.

\begin{prop}\label{Nerve-lemma}\cite[Corollary 4G.3]{hatcher}
Let $X$ be a finite simplicial complex covered by a finite family of subcomplexes $\mathcal{U} = \{ X_i , i \in I\}$  such that for all nonempty $J \subseteq I$, the intersection $\bigcap_{j \in J} X_j$ is either empty or contractible. Then $X$ is homotopically equivalent to the nerve $N(\mathcal{U})$.
\end{prop}

Vanishing results in this article will be obtained using the following particular case of the Nerve lemma.

\begin{lem}\label{contractible-union}
Let $X$ be a finite simplicial complex covered by $\{ X_i , i \in I\}$ a finite family of subcomplexes such that $\bigcap_{j \in J} X_j$ is contractible for all nonempty $J \subseteq I$. Then $X = \bigcup_{i \in I} X_i$ is contractible.
\end{lem}

\begin{proof}
By Proposition \ref{Nerve-lemma}, the space $X$ is homotopically equivalent to the nerve of the cover $\{ X_i , i \in I\}$, which is just a simplex.
\end{proof}

We recall our setting. Let $(W,S)$ be a Coxeter system and $\D$ be a simplex of dimension $|S| - 1$. Let $D$ be the Davis chamber of $(W,S)$. We defined for $\s \subseteq \D$, the subset $\D_\s$ to be the union of the faces of $\D$ not containing $\s$ and $D_\s = D \cap \D_\s$. Notice that $\D_\s$ is always a union of codimension 1 faces of $\D$. More precisely, if $\D_s$ is the codimension 1 face of $\D$ corresponding to the generator $s \in S$ and $\s = \bigcap_{s \in J}\D_s$, then $\D_\s = \bigcup _{s \in J^c}\D_s$ and $D_\s = \bigcup _{s \in J^c}D_s$. We may apply the Nerve lemma to the union $\bigcup _{s \in J^c}D_s$ in view of the following remark.

\begin{rem}\label{D_s-is-a-cone}
If $\s$ is the simplex in $\D$ corresponding to the parabolic subgroup $W_J$, then $D \cap \s$ is the geometric realization of the flag complex of finite parabolic subgroups containing $W_J$. Hence, if $D \cap \s$ is non-empty, then $D \cap \s$ is a cone with apex the barycenter of $\s$ and thus $D \cap \s$ is contractible. Notice that $D \cap \s$ is non-empty if and only if $W_J$ is finite.
\end{rem}

The previous lemma takes the following form in our setting.

\begin{cor}\label{contractible-union-of-D_s} For $\s \subseteq \D$, we write $\s = \bigcap_{s \in J}\D_s$ for some $J \subseteq S$, so that $D_\s = \bigcup_{s \in J^c} D_s$. If $W_{J^c} = \langle J^c \rangle$ is finite (or equivalently, if $ \bigcap_{s \in J^c} D_s$ is non-empty), then $D_\s$ is contractible.

\end{cor}

\begin{proof}
If $ \bigcap_{s \in J^c} D_s$ is non-empty, then every sub-intersection $\bigcap_{s \in J'} D_s$ is non-empty for $J' \subseteq J^c$, hence contractible by the previous remark. Now the result follows from Lemma \ref{contractible-union}.
\end{proof}

We wish to compute spaces $\Tilde{H}^{*}(D_\s)$ for $\s \subseteq \D$ that appear in the formula of Theorem \ref{formuleBetti}. The sum ranges over $\s$ with compact links, that is, over $\s$ with finite corresponding parabolic subgroups $W_\s$. What we said allows us to focus on a smaller class of simplices $\s$. Thus, the sum in the formula of Theorem \ref{formuleBetti} reduces to the following.

\begin{prop}
Let $(X, G)$ be a building in $\mathcal{B}t$ of thickness $q +1$. Then the $L^2$-Betti numbers of $G$ are given by
\begin{equation*}
    \b^k(G) = \sum_{\s \subseteq \D} \dim_\co \Tilde{H}^{k-1}(D_\s; \co) \cdot \dim_G L^2(G)^{\s}.
\end{equation*}
Moreover, the sum can be taken over simplices $\s$ corresponding to finite parabolic subgroups $W_J$ such that $W_{J^c}$ is infinite.
\end{prop}

\begin{proof}
The proof of Theorem \ref{formuleBetti} in \cite[8.5]{dymarajanuszkiewicz} already shows we can restrict ourselves to those $\s$ whose corresponding parabolic subgroup $W_J$ is finite (the argument is that $L^2(G)^\s \subseteq L^2(G)^{G_\s}$ and if the link of $\s$ is non-compact, then $ L^2(G)^{G_\s} = \{0 \}$). Now suppose $\s \subseteq \D$ corresponds to a finite parabolic subgroup $W_J$ and that the parabolic subgroup $W_{J^c}$ is also finite. Then by Corollary \ref{contractible-union-of-D_s}, the space $D_\s$ is contractible, so its cohomology does not contribute to the sum in any degree.
\end{proof}

We now turn to non-vanishing phenomena. In order to show Theorem \ref{infinite_family}, we will use the following non-vanishing criterion obtained using Proposition \ref{Nerve-lemma}.

\begin{cor}\label{affine_non-vanishing}
For $\s \subseteq \D$, we write $\s = \bigcap_{s \in J}\D_s$ for some $J \subseteq S$, so that $D_\s = \bigcup_{s \in J^c} D_s$.
If $W_{J^c} = \langle J^c \rangle$ is infinite, but every proper parabolic subgroup of $W_{J^c}$ is finite (or equivalently, if $ \bigcap_{s \in J'} D_s \neq \emptyset$ for all nonempty $J' \subsetneq J^c$ but $ \bigcap_{s \in J^c} D_s = \emptyset$), then $D_\s$ is homotopy equivalent to a sphere of dimension $(|I|-2)$.
\end{cor}

\begin{proof}
The nerve of the cover $\{ X_i , i \in I\}$ is the boundary of a simplex of dimension $(|I|-1)$, which homotopically is a sphere of dimension $(|I|-2)$. Proposition \ref{Nerve-lemma} gives the result.
\end{proof}

\begin{rem}
Infinite irreducible Coxeter groups such that every proper parabolic subgroup is finite are classified: they are exactly affine and compact hyperbolic Coxeter groups \cite[p.133, exercice 14]{bourbaki}. Compact hyperbolic Coxeter groups have rank $\leq 5$ \cite[p.133, exercice 15.c]{bourbaki}, so in higher rank the only examples are affine Coxeter groups \cite[p.100, Proposition 10]{bourbaki}.
\end{rem}

The following result summarizes the main idea, that is, we have non-vanishing of an $L^2$-Betti number in high degree for groups acting on buildings whose Weyl group is obtained as a perturbation of an affine Coxeter group by a finite Coxeter group.

\begin{thm}\label{S=Sph+Aff}
Let $(X, G)$ be a building in $\mathcal{B}t$ of thickness $q +1$ and irreducible Weyl group $(W, S)$. Suppose $S$ admits a partition $S = J_\mathrm{sph} \sqcup J_\mathrm{aff}$ such that $ W_\mathrm{sph} = \langle J_\mathrm{sph} \rangle$ is finite and $ W_\mathrm{aff} = \langle J_\mathrm{aff} \rangle$ is an infinite affine Coxeter group. Put $|J_\mathrm{aff}| = n+1$. Then, for $q$ large enough,
\begin{equation*}
    \b^n(G) > 0.
\end{equation*}
\end{thm}

\begin{proof}
Let $\s$ be the simplex corresponding to the subgroup $W_\mathrm{sph}$. Hence, $D_{\s} = \bigcup_{s \in J_\mathrm{aff}} D_s$ and thus by \ref{affine_non-vanishing}, the space $\Tilde{H}^*(D_{\s})$ is non-zero in degree $n-1$. Recall that for $q > 2^{|S|}$, we have $\dim_G L^2(G)^{\s} > 0$ \cite[p. 612]{dymarajanuszkiewicz}. Therefore by Theorem \ref{formuleBetti} we have
\begin{equation*}
    \b^n(G) \geq \dim \Tilde{H}^{n-1}(D_{\s}) \dim_G L^2(G)^{\s} > 0.
\end{equation*}
\end{proof}

\subsection{Application to measure equivalence of Kac-Moody lattices}

We now apply the previous result to $\ell^2$-Betti numbers of Kac-Moody groups. Recall from Section 2.2 that we can suppose our Kac-Moody groups to be center-free without adding conditions on the generalized Cartan matrix.

\begin{cor}\label{KM_lattice_Betti}
Let $\L$ be a center-free Kac-Moody group over a finite field $\mathbb{F}_q$ with irreducible Weyl group $(W,S)$ as in Theorem \ref{S=Sph+Aff}, we set $S = J_\mathrm{sph} \sqcup J_\mathrm{aff}$ with $J_\mathrm{sph} \neq \emptyset$. Suppose also that its Coxeter matrix has finite entries. Put $|J_\mathrm{aff}| = n+1$. For $q$ large enough,
\begin{equation*}
    \b^{k}(\L) \; \bigg \{ \; \begin{matrix}
    = 0  & $\textrm{if }$ & k \geq 2|S| - 1 \\
    > 0 &  $\textrm{if }$ & k = 2 n
    \end{matrix}
\end{equation*}
Moreover, $\L$ is infinite finitely presented, Kazhdan and simple.
\end{cor}

\begin{proof}
Let $G$ be the complete Kac-Moody group associated to $\L$. Theorem \ref{S=Sph+Aff} applies to $G$, so $\b^n(G) > 0$ for large $q$. Let $X$ be the building coming from the BN-pair of $G$. The group $G$ acts properly cocompactly on the Davis complex $X_D$. This complex has dimension $\leq |S|-1$, thus \cite[3.4]{dymarajanuszkiewicz} gives $H^k_{ct}(G, \rho) = 0$ for $k \geq |S|$ any quasi-complete representation $(\rho , V)$. Hence the Künneth formula for $L^2$-Betti numbers \cite[6.5]{petersen} gives for large $q$
\begin{equation*}
    \b^k(G \times G) \; \bigg \{ \; \begin{matrix}
    = 0  & $\textrm{if }$ & k \geq 2|S| - 1 \\
    > 0 &  $\textrm{if }$ & k = 2 n
    \end{matrix}
\end{equation*}
For $q > n+1$, the group $\L$ is a lattice in $G \times G$ by Theorem \ref{Kac-Moody} Assertion 1. By \cite[5.9]{petersen}, the sequences of $L^2$-Betti numbers of $\L$ and $G \times G$ are proportional. 

Since the entries of the Coxeter matrix of $(W,S)$ are finite, Theorem \ref{cohomG=0} tells us $G$ has property $(T)$ (for the same bound on $q$), thus so does $G \times G$ and any lattice in $G \times G$. The group $\L$ is finitely presented and simple in view of Theorem \ref{Kac-Moody} assertions 2, 3 and because we assumed $Z(\L) = \{ 0 \}$. 
\end{proof}

The previous corollary gives a control on the vanishing of $\ell^2$-Betti numbers for simple Kac-Moody lattices. We can now prove the theorem stated in the introduction using Gaboriau's projective invariance of $\ell^2$-Betti numbers by measure equivalence \cite[6.3]{gaboriau}.

\begin{proof}[Proof of Theorem \ref{infinite_family}] It suffices to take an affine diagram with $n+1$ generators $s_1, \ldots, s_{n+1}$, add a generator $s_0$, declare that it does not commute with at least one $s_i, i \geq 1$ and that the products $s_0 s_i$ have order $\leq 6$. Thus the Coxeter system $(W,S)$, where $S_\mathrm{sph} = \{ s_0 \}$, $S_\mathrm{aff} = \{s_1, \ldots, s_{n+1}\}$, $S = S_\mathrm{sph} \sqcup S_\mathrm{aff}$ and $W = \langle S \rangle$, satisfies the conditions of the previous corollary with $|S| = n+2$.

Let $\L_n$ be a Kac-Moody group over $\mathbb{F}_q$ with Weyl group $(W,S)$ with $q$ as in the corollary (its Coxeter matrix comes from a generalized Cartan matrix because of the hypothesis on the order of the products $s_0s_i$). Then
\begin{equation*}
    \b^k(\L_n) \; \bigg \{ \; \begin{matrix}
    = 0  & $\textrm{if }$ & k \geq 2n + 3, \\
    > 0 &  $\textrm{if }$ & k = 2 n.
    \end{matrix}
\end{equation*}
Hence the groups $(\L_{2n})$ have non-proportional sequences of $\ell^2$-Betti numbers, hence they are pairwise non measure equivalent in view of \cite[6.3]{gaboriau}.
\end{proof}

\section{An explicit family of Coxeter diagrams}

In this section we deal with a concrete family of Coxeter diagrams having a decomposition as in Theorem \ref{S=Sph+Aff}. The first aim is to exhibit a concrete family as before. However, in this example we can say more: the previous arguments compute all $L^2$-Betti numbers of a pair $(X,G)$ in $\mathcal{B}t$ with Weyl groups corresponding to these particular diagrams.
Let $(W,S)$ be the Coxeter system defined by the diagram $\widetilde{A}_{n,2}$ (with $n+1$ generators and $n \geq 3$) as below:

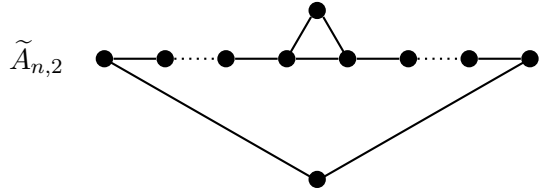
\begin{center}
  \begin{tikzpicture}[scale=.4]
    \draw (-1,0) node[anchor=east]  {$\widetilde{A}_{n,2}$};
    \foreach \x in {0,...,7}
    \draw[xshift=\x cm,thick, fill=black] (\x cm,0) circle (.25cm);
    \draw[xshift = 3.5 cm,thick, fill=black] (3.5 cm,1.6) circle (.25cm);
    \draw[xshift = 3.5 cm,thick, fill=black] (3.5 cm,- 4) circle (.25cm);
    
    \foreach \y in {0.15, 2.15,3.15,4.15, 6.15}
    \draw[xshift=\y cm,thick] (\y cm,0) -- +(1.4 cm,0);
    \foreach \y in {1.15, 5.15}
    \draw[xshift=\y cm,dotted, thick] (\y cm,0) -- +(1.4 cm,0);
    \draw[xshift=6 cm,thick] (30: 0 mm) -- (60: 16 mm);
    \draw[xshift=8 cm,thick] (30: 0 mm) -- (120: 16 mm);
    \draw[xshift=0 cm,thick] (30: 0 mm) -- (-30: 80 mm);
    \draw[xshift=14 cm,thick] (30: 0 mm) -- (-150: 80 mm);

  \end{tikzpicture}
\end{center}

\subsection{Maximal finite parabolic subgroups}

Let $s_1, \ldots, s_n$ be the generators corresponding to the affine subgroup of type $\widetilde{A}_{n-1}$ in $W$ and $s_0$ be the remaining generator, so that  $\langle s_0, s_1, s_2 \rangle$ is an infinite affine Coxeter group of type $\widetilde{A}_{2}$. To obtain finite parabolic subgroups one has to consider subsets $J \subset S$ that do not contain these three generators simultaneously.

For simplicity, call $W_i = \langle S \setminus \{ s_i \} \rangle$ and for $i \neq j$ call $W_{i,j} = \langle S \setminus \{ s_i , s_j \} \rangle$. The subgroups $W_1$ and $W_2$ are of type $A_{n-1}$, thus finite. Therefore, they are maximal finite parabolic subgroups of $W$. The subgroup $W_0$ is of type $\Tilde{A}_{n-1}$, thus infinite. It is affine, so every proper parabolic subgroup of $W_0$ is finite. Therefore the parabolic subgroups $W_{0, i}$ for $1 \leq i \leq n$ are maximal finite.

\begin{rem}
We can proceed in the same way for Kac-Moody groups with Weyl groups coming from the same alteration of an affine Coxeter diagram.

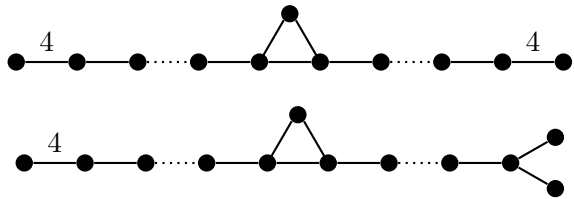
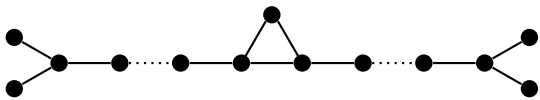
\begin{center}
  \begin{tikzpicture}[scale=.4]
    \draw (-1,0) node[anchor=east]{} ;
    \foreach \x in {0,...,9}
    \draw[xshift=\x cm,thick, fill=black] (\x cm,0) circle (.25cm);
    \draw[xshift = 4.5 cm,thick, fill=black] (4.5 cm,1.6) circle (.25cm);
    
    \foreach \y in {1.15, 3.15,4.15,5.15, 7.15}
    \draw[xshift=\y cm,thick] (\y cm,0) -- +(1.4 cm,0) ;
    \draw[xshift=0.15 cm,thick] (0.15 cm,0) -- node[above]{4} +(1.4 cm,0) ;
    \draw[xshift=8.15 cm,thick] (8.15 cm,0) -- node[above]{4} +(1.4 cm,0) ;
    \foreach \y in {2.15, 6.15}
    \draw[xshift=\y cm,dotted, thick] (\y cm,0) -- +(1.4 cm,0) ;
    \draw[xshift=8 cm,thick] (30: 0 mm) -- (60: 16 mm);
    \draw[xshift=10 cm,thick] (30: 0 mm) -- (120: 16 mm);
  \end{tikzpicture}
\end{center}

\begin{center}
  \begin{tikzpicture}[scale=.4]
    \draw (-1,0) node[anchor=east]{} ;
    \foreach \x in {0,...,8}
    \draw[xshift=\x cm,thick, fill=black] (\x cm,0) circle (.25cm);
    \draw[xshift = 4.5 cm,thick, fill=black] (4.5 cm,1.6) circle (.25cm);
    \draw[xshift=16 cm,thick, fill=black] (30: 17 mm) circle (.25cm);
    \draw[xshift=16 cm,thick, fill=black] (-30: 17 mm) circle (.25cm);
    
    \foreach \y in {1.15, 3.15,4.15,5.15, 7.15}
    \draw[xshift=\y cm,thick] (\y cm,0) -- +(1.4 cm,0) ;
    \draw[xshift=0.15 cm,thick] (0.15 cm,0) -- node[above]{4} +(1.4 cm,0) ;
    \foreach \y in {2.15, 6.15}
    \draw[xshift=\y cm,dotted, thick] (\y cm,0) -- +(1.4 cm,0) ;
    \draw[xshift=8 cm,thick] (30: 0 mm) -- (60: 16 mm);
    \draw[xshift=10 cm,thick] (30: 0 mm) -- (120: 16 mm);
    \draw[xshift=16 cm,thick] (30: 3 mm) -- (30: 14 mm);
    \draw[xshift=16 cm,thick] (-30: 3 mm) -- (-30: 14 mm);
  \end{tikzpicture}
\end{center}

\begin{center}
  \begin{tikzpicture}[scale=.4]
    \draw (-1,0) node[anchor=east]{} ;
    \foreach \x in {1,...,8}
    \draw[xshift=\x cm,thick, fill=black] (\x cm,0) circle (.25cm);
    \draw[xshift = 4.5 cm,thick, fill=black] (4.5 cm,1.6) circle (.25cm);
    \draw[xshift=16 cm,thick, fill=black] (30: 17 mm) circle (.25cm);
    \draw[xshift=16 cm,thick, fill=black] (-30: 17 mm) circle (.25cm);
    \draw[xshift=2 cm,thick, fill=black] (210: 17 mm) circle (.25cm);
    \draw[xshift=2 cm,thick, fill=black] (150: 17 mm) circle (.25cm);
    
    \foreach \y in {1.15, 3.15,4.15,5.15, 7.15}
    \draw[xshift=\y cm,thick] (\y cm,0) -- +(1.4 cm,0) ;
    \foreach \y in {2.15, 6.15}
    \draw[xshift=\y cm,dotted, thick] (\y cm,0) -- +(1.4 cm,0) ;
    \draw[xshift=8 cm,thick] (30: 0 mm) -- (60: 16 mm);
    \draw[xshift=10 cm,thick] (30: 0 mm) -- (120: 16 mm);
    \draw[xshift=16 cm,thick] (30: 3 mm) -- (30: 14 mm);
    \draw[xshift=16 cm,thick] (-30: 3 mm) -- (-30: 14 mm);
    \draw[xshift=2 cm,thick] (150: 3 mm) -- (150: 14 mm);
    \draw[xshift=2 cm,thick] (210: 3 mm) -- (210: 14 mm);
  \end{tikzpicture}
\end{center}

In each case, the structure of maximal finite parabolic subgroups is the same, therefore the following results for $W$ remain the same.

\end{rem}

\subsection{Cohomology of subcomplexes}

The following result completes the computation of the cohomology of $D_\s$ for all $\s \subseteq \D$. More precisely, the simplices $\s$ appearing in the following theorem are exactly those whose corresponding parabolic subgroup $W_J$ of $W$ is finite but with $W_{J^c}$ infinite. Let $\D_i$ be the face of codimension 1 of $\D$ corresponding to the generator $s_i$ of $W$ and $D_i = D \cap \D_i$.

\begin{thm} Let $D$ be the Davis chamber of $(W,S)$ of type $\Tilde{A}_{n,2}$. \\
$1.$ For $\s = \D_0$, the space $D_\s = \bigcup_{j \neq 0} D_{j}$ has the cohomology of an $(n-2)$-dimensional sphere. \\
$2.$ For $\s = \D$, the space $D_\s = D \cap \partial \D$ has the cohomology of an $(n-2)$-dimensional sphere. \\
Let $\t \subset \D$ be the simplex such that $\D_\t = \D_0 \cup \D_1 \cup \D_2$.\\
$3.$  The space $D_\t$ has the cohomology of the circle.\\
$4.$ For $\t \subsetneq \s \subsetneq \D$, the space $D_\s$ is contractible.
\end{thm}

\begin{proof}
Assertions 1 and 3 follow from \ref{affine_non-vanishing} since $W_{\D_0}$ and $\langle s_0, s_1, s_2 \rangle$ are infinite affine Coxeter groups of type $\Tilde{A}_{n-1}$ and $\Tilde{A}_{2}$. It remains to prove assertions 2 and 4, that is, the cases $\t \subsetneq \s \subseteq \D$. Denote $A_I = D_1 \cup D_2 \cup \bigcup_{k \in I} D_k$, for the corresponding nonempty $I \subseteq \{3, \ldots, n\}$ so that $D_\s = D_0 \cup A_I$. Denote $D_{ij} = D_i \cap D_j$ and $D_{ijk} = D_i \cap D_j \cap D_k$. Our first goal is to prove that the intersection $D_0 \cap A_I$ is contractible, write it as the following union:
\begin{equation*}
    D_0 \cap A_I =  D_{0,1} \cup D_{0,2} \cup  \bigcup_{l \in I} D_{0,l}.
\end{equation*}
First the union $A_I' = \bigcup_{l \in I} D_{0,l}$ is contractible because of \ref{contractible-union} (the group $\langle s_0, s_i | i \in I \rangle$ is finite, so the same argument as in the proof of \ref{contractible-union-of-D_s} works). Again by \ref{contractible-union}, the intersections $A_I' \cap D_{0,1} =  \bigcup_{l \in I} D_{0,1,l}$ and $A_I' \cap D_{0,2}  =  \bigcup_{l \in I} D_{0,2,l}$ are contractible (the groups $\langle s_0,s_1, s_i | i \in I \rangle \subseteq W_1$ and $\langle s_0,s_2, s_i | i \in I \rangle \subseteq W_2$ are finite). Notice that $D_{0,1} \cap D_{0,2} = (\D_0 \cap \D_1 \cap \D_2) \cap D = \emptyset$. By the Nerve lemma, $D_0 \cap A_I$ is homotopy equivalent to the nerve of the cover $\mathcal{U} = \{  D_{0,1}, D_{0,2},  A_I' \}$, which is the barycentric subdivision of a simplex of dimension 1. Thus the intersection $D_0 \cap A_I$ is contractible.

Therefore Mayer-Vietoris tells us that $D_0 \cup A_I$ has the same cohomology as $A_I$. If $I = \{3, \ldots, n\}$, we have $A_I = D_{\D_0}$ so from the first assertion we know $A_I$ has the cohomology of an $(n-2)$-dimensional sphere. Thus $D_\D = D \cap \partial \D = D_0 \cup A_I$ has the cohomology of an $(n-2)$-dimensional sphere. If $I \subsetneq \{3, \ldots, n\}$, then $A_I$ is contractible because of \ref{contractible-union}, hence the union $D_\s= D_0 \cup A_I$ is contractible.

\end{proof}

We now recover the corresponding results for $L^2$-Betti numbers of groups in $\mathcal{B}t$ with Weyl group $\Tilde{A}_{n,2}$.

\begin{cor}
Let $(X, G)$ be a building in $\mathcal{B}t$ of thickness $q + 1$ and Weyl group of type $\Tilde{A}_{n,2}$ with $n \geq 3$. Normalize the Haar measure $\mu$ on $G$ so that the stabilizer $G_\D$ of an alcove $\D$ has measure 1. Then we have:

\begin{equation*}
    \b^k(G) =  \; \Bigg \{ \; \begin{matrix}
     \dim_G L^2(G)^{\D} + \dim_G L^2(G)^{\D_0}  &  k = n -1, \\
     \dim_G L^2(G)^{\t} &  k = 2, \\
     0 &  \textrm{ otherwise. } 
    
    \end{matrix}
\end{equation*}
\end{cor}

\begin{proof} The sum in the formula \ref{formuleBetti} runs over $\s \subseteq \D$ such that its corresponding parabolic subgroup $W_J, J\subseteq S$ is finite and such that $W_{J^c}$ is not finite. Such simplices are exactly the ones treated in the theorem. Their non-vanishing cohomology groups give the result.
\end{proof}

Thus, by the same arguments as in the previous section, the Künneth formula gives the following more precise statement for Kac-Moody lattices $\L$ having Weyl group of type $\Tilde{A}_{n,2}$. Moreover, the theorems we use are quantitative: they give us an explicit formula for $\b^k(\L)$. We do not give the formulas since the only information that matters to us is in which degrees $\b^k(\L)$ vanishes and it which it does not vanish for large $q$.

\begin{cor}
Let $\L$ be a center-free Kac-Moody group over a finite field $\mathbb{F}_q$ with Weyl group of type $\Tilde{A}_{n,2}$. For $q$ large enough,
\begin{equation*}
    \b^{k}(\L) \; \bigg \{ \; \begin{matrix}
    > 0  & k \in \{ 4, n+1, 2n-2 \},  \\
    = 0 &  $\textrm{ otherwise}$ .
    \end{matrix}
\end{equation*}
Moreover, $\L$ is infinite finitely presented, Kazhdan and simple.
\end{cor}

\section{Sphericity and cohomological finiteness}

In this section we discuss connections between $n$-sphericity and finiteness properties $F_n$ and $FP_n$ of a Kac-Moody group (over a finite field). Theorem \ref{Kac-Moody} 3, says that for large $q$, the 2-sphericity condition implies property $F_2$, that is, finite presentation.
The converse is still a conjecture, but it has at least been proven in particular cases \cite{abramenko-gates}. Much less is known for higher finiteness properties. Abramenko obtained some partial results in this direction \cite{abramenkofiniteness}.

Discrete Kac-Moody groups with Weyl group of type $\widetilde{A}_{n,2}$ are finitely presented (at least for $q > 6$) since they are $2$-spherical, but they are not $3$-spherical because of the subdiagram of type $\Tilde{A}_2$ that we introduced. We can ask if it is possible to obtain stronger finiteness properties for non-affine Kac-Moody groups. Here we present a family of non-affine Coxeter diagrams that are $8$-spherical but not $9$-spherical. We call $\widetilde{B}_{n,8}$ the Coxeter diagram with $n+1$ generators as below:

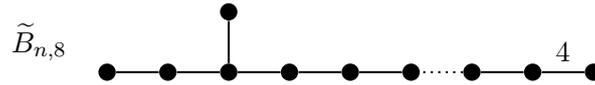
\begin{center}
  \begin{tikzpicture}[scale=.4]
    \draw (-1,1) node[anchor=east]  {$\widetilde{B}_{n,8}$};
    \foreach \x in {0,...,8}
    \draw[thick,xshift=\x cm,fill=black] (\x cm,0) circle (2.5 mm);
    \draw[thick, fill=black] (4 cm,2 cm) circle (2.5 mm);
    
    \foreach \y in {0,...,4}
    \draw[thick,xshift=\y cm] (\y cm,0) ++(.3 cm, 0) -- +(14 mm,0);
    \draw[thick] (4 cm, 3mm) -- +(0, 1.4 cm);
    \draw[xshift=5.2 cm,dotted, thick] (5.2 cm,0) -- +(1.4 cm,0) ;
    \draw[thick,xshift=6 cm] (6 cm,0) ++(.3 cm, 0) -- +(14 mm,0);
    \draw[thick,xshift=7 cm] (7 cm,0) ++(.3 cm, 0) -- node[above]{4} +(14 mm,0);
  \end{tikzpicture}
\end{center}

The Davis chamber of a Coxeter group of type $\widetilde{B}_{n,8}$ is similar to that of $\widetilde{A}_{n,2}$. It contains a subdiagram of type $\Tilde{E}_8$ when $n \geq 9$. If we remove the generator $s_0$ at the left, we obtain an affine subdigram of type $\Tilde{B}_{n-1}$ with generators $s_1, \ldots s_n$. If we take out any other generator $s_i$ ($i = 1, \ldots, 8$) of $\Tilde{E}_8$, we obtain a maximal finite parabolic subgroup. Hence the Davis chamber consists of an affine part ($n$ cubes of dimension $n-1$) and $8$ cubes of dimension $n$. The result for the diagram $\Tilde{A}_{n,2}$ remains valid for $\widetilde{B}_{n,8}$.

\begin{thm} Let $D$ be the Davis chamber of $(W,S)$ of type $\Tilde{B}_{n,8}$. \\
$1.$ For $\s = \D_0$, $D_\s = \bigcup_{j \neq 0} D_{s_j}$ has the cohomology of an $(n-2)$-dimensional sphere. \\
$2.$ For $\s = \D$, $D_\s = D \cap \partial \D$ has the cohomology of an $(n-2)$-dimensional sphere. \\
$3.$ Let $\t \subset \D$ be the simplex such that $\D_\t = \bigcup_{i = 0}^8 \D_i$. The space $D_\t$ has the cohomology of a $7$-dimensional sphere.\\
$4.$ For $\t \subsetneq \s \subsetneq \D$, $D_\s$ is contractible.
\end{thm}

\begin{proof}
The proof is almost the same as for $\Tilde{A}_{n,2}$: assertions 1. and 3. follow from \ref{affine_non-vanishing} while 2. and 4. need some adjustment. We keep the same notations as in the other proof. The only difference is that in order to show $D_0 \cup A_I$ has the same cohomology as $A_I$ for some non-empty $I \subseteq \{9, \ldots, n\}$, one has to show that the nerve of some cover of $D_0 \cap A_I$ by 9 elements is the boundary of a simplex of dimension 8 minus one face of codimension 1, showing that $D_0 \cap A_I$ is contractible.  
\end{proof}

Thus we obtain the same corollaries for Kac-Moody groups as in the previous section. The complete Kac-Moody groups with Weyl group of type $\Tilde{B}_{n,8}$ have non-vanishing $L^2$-Betti numbers (for large thickness) exactly in degrees $n-1$ and $8$. At the level of discrete Kac-Moody groups, we obtain an infinite family of infinite finitely presented, Kazhdan, $8$-spherical, simple groups that are pairwise non-measure equivalent. By the Künneth formula, their $\ell^2$-Betti numbers vanish except in degrees $16, n+7$ and $2n-2$.

One cannot have better sphericity properties for a non-affine irreducible Coxeter diagram. The following proposition has to be stated somewhere. The author did not find a reference, so we give the proof here.

\begin{prop}
A $9$-spherical irreducible Coxeter group is either finite or affine.
\end{prop}

\begin{proof}
The proof consists of ruling out all possibilities by looking at the classification of finite and affine Coxeter groups. More precisely, we look at two families of integers: the valencies of vertices of the Coxeter diagram as a graph and the labelling of edges of the diagram.

Let $(W,S)$ be a $9$-spherical irreducible Coxeter system such that $|S| \geq 9$. Its Coxeter diagram is connected by irreducibility.
If the valency of every vertex of the diagram is $\leq 2$, then the Coxeter diagram of $W$ without labelling is of type $A_n$ or $\Tilde{A}_n$ with $n \geq 9$. If we label edges of a diagram of type $A_n$ by numbers $\geq 4$, the $5$-sphericity of $W$ rules out all possibilities, except having extremal edges labelled by $4$: the possible diagrams for $W$ are $A_n$, $B_n = C_n$ or $\Tilde{C}_n$. If we label an edge of a diagram of type $\Tilde{A}_n$ by a number $\geq 4$, the associated group is not $5$-spherical, hence the only possibility for $W$ in this case is to have a diagram of type $\Tilde{A}_n$.

If not, there exists a vertex $y$ of valency $\geq 3$. The valency of $y$ cannot be $\geq 4$ since this would give directly an infinite subgroup of rank 5. Frow now on $y$ has valency 3. Call $y_1, y_2$ and $y_3$ the three neighbors of $y$. None of these vertices can have valency $\geq 3$ since this would give rise to an infinite subgroup of rank 6. The three neighbors cannot have valency $\geq 2$ simultaneously since the Coxeter graph would contain a subgraph of type $\Tilde{E}_6$ that corresponds to an infinite subgroup of rank 7. Thus we can suppose $y_3$ has valency 1. 

If both $y_1$ and $y_2$ have valency 2, then there is a subdiagram of type $\Tilde{E}_7$ or of type $\Tilde{E}_8$, which respectively correspond to infinite subgroups of rank 8 and 9. Again, $9$-sphericity rules out these possibilities.

We may assume $y_2$ and $y_3$ have valency 1, and $y_1$ has valency $\geq 2$. This implies $W$ is of type $\Tilde{B}_n$ or $\Tilde{D}_n$. Indeed, the graph of $W$ has to contain $\Tilde{B}_n$ or $\Tilde{D}_n$ as a subgraph, with possible vertices of higher valency. If such a vertex exists, then again there is a subdiagram of type $\Tilde{E}_7$ or of type $\Tilde{E}_8$.
\end{proof}

\begin{rem}
In view of this proposition, Theorem \ref{cohomG=0} does not tell us anything about cohomological vanishing in degrees $\geq 8$ for groups with BN-pair acting on a building of finite thickness with non-affine Weyl group. When the Weyl group is affine, the same vanishing result was obtained by Garland in \cite{garland}.
\end{rem}

\begin{rem}
In \cite{SWZ}, it is used that properties $F_n$ are invariant by quasi-isometries \cite{alonso}. They construct finitely presented simple groups that are $F_{n-1}$ but not $F_n$ for each $n$. This gives an infinite family of infinite finitely presented simple groups that are pairwise not quasi-isometric. 

If one could prove that property $F_n$ implies $n$-sphericity for a Kac-Moody group, then the previous proposition shows that every non-affine Kac-Moody group over a finite field is at most $F_8$. Thus the method of \cite{SWZ} could not work for non-affine Kac-Moody groups if the previous conjecture is true.
\end{rem}

\bibliographystyle{amsalpha}
\bibliography{refs.bib}

\noindent Antonio López Neumann \\
 CMLS, CNRS, École polytechnique, Institut polytechnique de Paris, \\ 91128 Palaiseau Cedex, (France) \\
 antonio.lopez-neumann@polytechnique.edu \\
 
 \noindent This research is part of the author's PhD thesis at École polytechnique under the supervision of Marc Bourdon and Bertrand Rémy.

\end{document}